\documentclass{amsart}

\usepackage{amsfonts}
  \usepackage{amssymb}
 \newtheorem{theorem}{Theorem}
 \newtheorem{lemma}{Lemma}
  
 \newtheorem{pro}{Proposition}
 
 \newtheorem{remark}{Remark}
 
 \newcommand{\bee}[1]{\begin{equation}\label{#1}}
 \newcommand{\ene}{\end{equation}}

\begin{document}
\title{An uncountable family of almost nilpotent varieties of polynomial growth}
\author[Mishchenko]{S. Mishchenko}\address{Department of Applied Mathematics
\\ Ulyanovsk  State University \\
Ulyanovsk 432970, Russia} \email{mishchenkosp@mail.ru}

\author{A. Valenti}\address{Dipartimento di Energia, ingegneria dell'Informazione e Modelli Matematici \\ Universit\`a di Palermo \\
90128 Palermo, Italy}
 \email{angela.valenti@unipa.it}

\date{}

\begin{abstract}
A non-nilpotent variety of algebras is almost nilpotent if any
proper subvariety is nilpotent. Let the base field be of
characteristic zero. It has been shown that for associative or Lie
algebras only one such variety exists. Here we present infinite
families of such varieties. More precisely we shall prove the
existence of

1) a countable family of almost nilpotent varieties of at most linear growth and

2) an uncountable family of almost nilpotent varieties of at most quadratic growth.
\end{abstract}

\maketitle

\section{Introduction}

Let $F$  be a field of characteristic zero and $F\{X\}$ the free non
associative algebra on a countable set $X$ over $F$. If
$\mathcal{V}$ is a variety of non necessarily associative algebras
and $Id(\mathcal{V})$ is the $T$-ideal of polynomial identities of
${\mathcal{V}}$, then $F\{ X \}/Id(\mathcal{V})$ is the relatively
free algebra of countable rank of the variety ${\mathcal{V}}$. It is
well known that in characteristic zero every identity is equivalent
to a system of multilinear ones, and an important invariant is
provided by the sequence of dimensions $c_n({\mathcal{V}})$ of the
$n$-multilinear part of $F\{ X \}/Id(\mathcal{V})$, $n=1,2,\ldots$.
More precisely,
for every $n\ge 1$ let $P_n$ be the space of multilinear polynomials
in the variables $x_1,\ldots, x_n.$ Since char $F=0$,
$F\{ X\} /Id(\mathcal{V})$ is determined by the
sequence of subspaces $\{P_n /(P_n \cap Id(\mathcal{V}))\}_{n\ge
1}$ and the integer $c_n(\mathcal{V}) =\dim P_n /(P_n \cap
Id(\mathcal{V}))$ is called the $n$-th codimension of $\mathcal{V}$.
The growth function determined by the sequence of integers $\{
c_n({\mathcal{V}})\}_{n\ge 1}$ is the growth of the variety
${\mathcal{V}}.$

In general a variety ${\mathcal{V}}$ has overexponential growth, i.e.,
the sequence of codimensions cannot be bounded by any exponential function.
Recall that ${\mathcal{V}}$ has exponential growth if
$c_n(\mathcal{V})\le a^n$, for all $n\ge 1$, for some constant $a$.
For instance any variety generated by a finite dimensional algebra has exponential growth.
For such varieties the limit $\lim_{n\to\infty}\sqrt[n]{c_n(\mathcal{V})}=exp(\mathcal{V})$,
is called the PI-exponent of the variety $\mathcal{V}$, provided it exists.

We say that a variety $\mathcal{V}$ has polynomial growth if there exist
constants $\alpha, t \ge 0$ such that  asymptotically $c_n({\mathcal{V}})\simeq
\alpha n^t$.
When $t=1$ we speak of linear growth and when $t=2$, of quadratic growth.

Moreover ${\mathcal{V}}$ has intermediate growth if
for any $k> 0,$ $a>1$ there exist constants $C_1, C_2,$ such that
for any $n$ the inequalities
$$
C_1 n^k < c_n({\mathcal{V}}) < C_2 a^n
$$
hold.
Finally we say that a variety $\mathcal{V}$ has subexponential growth  if for any constant
 $B$ there exists $n_0$ such that for all $n > n_0,$   $c_n({\mathcal{V}}) < B^n.$
Clearly varieties with polynomial growth or intermediate growth have
subexponential growth and it can be shown that varieties realizing
each growth can be constructed. For instance a class  of varieties of
intermediate growth was constructed in  \cite{GMZapplied}.

The purpose of this note is the study of the almost nilpotent varieties.
Recall that a variety $\mathcal{V}$ is almost nilpotent if it is not
nilpotent but all proper subvarieties are nilpotent.

About previous results, if we consider varieties of associative
algebras, it is easily seen that the only almost nilpotent variety
is the variety $\mathcal{V}$ of commutative algebras, (the sequence
of codimensions is $c_n(\mathcal{V})=1, n\ge 1$). In case of
varieties of Lie algebras it has been shown that there is also only
one almost nilpotent variety: the variety ${\mathcal A}^2$ of
metabelian Lie algebras and in this case $c_n({\mathcal A}^2)=n-1$.
In \cite{frolova-shu} it was proved that there exist only two almost
nilpotent varieties of Leibniz algebras and both varieties have at
most linear growth. For general non associative algebras, in
\cite{mish-val-exp2} an almost nilpotent variety of exponent two was
constructed. Later in  \cite{mis-shu} it was proved that for any
integer $m$ an almost nilpotent variety with exponent $m$ exists.
Recently in \cite{mish} it was proved the existence of almost
nilpotent varieties with fractional exponent.

An algebra satisfying the identity $x(yz)\equiv 0$ will be called left nilpotent of index two.
In \cite{mish-val-subexp} two almost nilpotent varieties with linear growth were  constructed and it
was proved that they represent a full list of almost nilpotent varieties with subexponential growth in the class of left nilpotent algebras of index two.
For commutative (anticommutative) metabelian algebras similar result were obtained in \cite{ch-frolova-subexp}, \cite{mish-shu-subexp}.

The purpose of this note is to  prove the existence of two families
of almost nilpotent varieties. The first one is a countable family
of at most linear growth and the second one
 is an uncountable family of at most quadratic growth.

\section{The general setting}

Throughout $A$ will be a non necessarily associative algebra over a field
$F$ of characteristic zero and $F\{X\}$ the free non associative
algebra on a countable set $X=\{x_1,x_2, \ldots\}$. The polynomial
identities satisfied by $A$ form a T-ideal $Id(A)$ of $F\{X\}$ and
by the standard multilinearization process, we consider only the
multilinear polynomials lying in $Id(A)$. To this end, for every
$n\ge 1$,  we set $P_n$ to be the space of multilinear polynomials in
$x_1, \ldots, x_n$, and we let the symmetric group $S_n$ act on
$P_n$ be setting $\sigma f(x_1, \ldots,x_n)=f(x_{\sigma(1)}, \ldots,
x_{\sigma(n)})$, for $\sigma\in S_n, f\in P_n$.

The space $P_n(A)={P_n}/({P_n\cap Id(A)})$ has an induced structure
of $S_n$-module and we let $\chi_n(A)$ be its character, called the
$n$-th cocharacter of $A$. By complete reducibility we write
$$
\chi_n(A)=\sum_{\lambda\vdash n} m_\lambda \chi_\lambda
$$
where $\chi_\lambda$ is the irreducible $S_n$-character
corresponding to the partition $\lambda\vdash n$ and $m_\lambda\ge
0$ is the corresponding multiplicity (we refer the reader to
\cite{gz} for an account of this approach).

 \smallskip
 We next recall some
basic properties of the representation theory of the symmetric group
that we shall use in the sequel. Let $\lambda\vdash n$ and let
$T_\lambda$ be a Young tableau of shape $\lambda\vdash n$. We denote
by $e_{T_\lambda}$ the corresponding essential idempotent of the
group algebra $FS_n$. Recall that $e_{T_\lambda}= {R^+_{T_\lambda}}
{C^-_{T_\lambda}}$ where $ {R^+_{T_\lambda}}= \sum_{\sigma\in
R_{T_\lambda}}\sigma, \, {\rm and}\, {C^-_{T_\lambda}}=\sum_{\tau\in
C_{T_\lambda}}({\rm sgn} \tau )\tau $ and $R_{T_\lambda},$
$C_{T_\lambda}$ are the groups of row and column stabilizers of
$T_\lambda$, respectively. Recall that if $M_\lambda$ is an
irreducible $S_n$-submodule of $P_n(A)$ corresponding to $\lambda$,
there exists a polynomial $f(x_1,\dots ,x_n)\in P_n$ and a tableau
$T_\lambda$ such that $e_{T_\lambda}f(x_1,\dots ,x_n)\not\in Id(A).$
Let $e'_{T_\lambda}={C^-_{T_\lambda}}{R^+_{T_\lambda}}
{C^-_{T_\lambda}}.$ Since ${R^+_{T_\lambda}}
{C^-_{T_\lambda}}{R^+_{T_\lambda}}  {C^-_{T_\lambda}}\ne 0$ then
$e'_{T_\lambda}$ is a nonzero essential idempotent that generates the same
irreducible module and so also $e'_{T_\lambda}f(x_1,\dots
,x_n)\not\in Id(A).$

In what follows we shall also denote by  $g(\lambda)$  the
polynomial obtained from the essential idempotent corresponding to a
tableau of shape $\lambda$ by identifying the elements in each row.
Recall that $g(\lambda)$ is an highest weight vector of the general
linear group $GL_k(F)$ where $k$ is the number of distinct part of
$\lambda$ (see \cite{D})

Now, for a fixed arrangement of the parentheses $T,$ let us denote
by $P_n^T$ the subspace of $P_n$ spanned by the monomials whose
arrangement of the parentheses is $T$. Let also $P_n^T(A)=
P_n^T/(P_n^T\cap Id(A)).$ Then clearly $P_n(A) = \sum_{T} P_n^T(A).$

Since the $S_n$-module $P_n^T(A)$ is a homomorphic image of
$P_n^T\equiv FS_n$, the regular $S_n$-representation, it follows
that, if $\chi_n(A)^T$ is the $S_n$-character of $P_n^T(A)$, then
$$
\chi_n(A)^T=\sum_{\lambda\vdash n} m_\lambda^T \chi_\lambda
$$
and $m^T_\lambda \le d_\lambda =\deg\chi_\lambda.$ Clearly
$m_\lambda\le \sum_Tm_\lambda^T$.

Throughout we shall also use the following convention: we shall
write the same symbol (e.g.  $\bar{} \ ,\  \, \tilde{}$\ ) over two
or more variables of a polynomial to indicate that the polynomial is
alternating on these variables.

For instance  $x_3\overline x_1\overline x_2=x_3x_1x_2-x_3x_2x_1$.

\medskip

We also need to recall some results from the theory of infinite words (see \cite{lotha}).
Recall that, given an infinite (associative) word $w$ in the alphabet $\{0,1\}$
the complexity $\mbox{Comp}_w$ of $w$ is defined as the function $\mbox{Comp}_w:
\mathbb{N} \to \mathbb{N}$, where $\mbox{Comp}_w(n)$ is the number
of distinct subwords of $w$ of length $n$.

Also, an infinite word $w=w_1w_2\cdots $ is {\it periodic} with period $T$ if $w_i=w_{i+T}$ for
$i=1,2,\dots $. It is easy to see that for any such word
$\mbox{Comp}_w(n) \le T$. Moreover, an infinite word $w$ is called
a {\it Sturmian} word if $\mbox{Comp}_w(n)=n+1$ for all $n\ge 1$.

For a finite word $x$, the height $h(x)$ of $x$ is the number of
letters $1$  appearing in $x$. Also, if $|x|$ denotes the length
of the word $x$, the {\it slope}  of  $x$ is defined as $\pi(x)=
\frac{h(x)}{|x|}$. In some cases this definition can be extended
to infinite words as follows.
Let $w$ be some infinite
word and let $w(1,n)$ denote its prefix subword of length $n$. If
the sequence $\frac{h(w(1,n))}{n}$ converges for $n\to\infty$ and
the limit
$$
 \pi(w)= \displaystyle{\lim_{n\to\infty} \frac{h(w(1,n))}{n}}
$$
exists then $\pi(w)$ it is called the slope of $w$.
Examples of infinite words for which the slope is not defined can be given.
Nevertheless for periodic and Sturmian words the slope is well defined.
In the next proposition we reassume  the main properties of these words that we shall use here.

\begin{theorem} $($\cite[Section 2.2]{lotha}$)$ \label{lothaire} Let $w$ be a
Sturmian or periodic word. Then there exists a constant $C$ such
that
 \begin{itemize}
\item[1)] $|h(x)-h(y)| \le C,$ for any finite subwords $x,y$ of $w$
with $|x|=|y|$;
 \item[2)] the slope $\pi(w)$ of $w$ exists;
 \item[3)] $|\pi(u)-\pi(w)|\le \frac{C}{|u|}$, for any non-empty subword $u$ of $w$;
\item[4)] for any real number $\alpha\in (0,1)$ there exists
a word $w$ with $\pi(w)=\alpha$ and $w$ is  Sturmian or periodic
according as $\alpha$ is irrational or rational, respectively.
 \end{itemize}
If $w$ is Sturmian we can take $C=1$, and if $w$ is periodic
of period $t$, then $\pi(w)= \frac{h(w(1,t))}{t}$.
 \end{theorem}

\section{Algebras constructed from periodic or Sturmian words}

Our aim in this section is  to  prove the existence of two families
of almost nilpotent varieties. The first is  a countable family of
varieties of at most linear growth and the second is an uncountable
family of at most quadratic growth. To do this we will make use of
an algebra constructed in \cite{gia-mish}.

Throughout  $A$ will be the algebra generated by one element $a$ such that
every word in $A$ containing two or more subwords equal to $a^2$
must be zero.

Note that in particular  the algebra $A$  is metabelian, i.e., it
satisfies the identity
$$(x_1x_2)(x_3x_3)\equiv 0.
$$
A partial decomposition of the cocharacter of $A$ was given in
\cite{gia-mish} and we recall it here.

Let $L_a$ and $R_a$ denote the linear transformations on $A$ of
left and right multiplication by $a$, respectively. We shall usually
write $bL_a=L_a(b)=ab$ and $bR_a=R_a(b)=ba.$

 We have the following

\begin{remark}
\hskip3cm

\begin{itemize}
\item[1)]
 $\chi_n(A)= m_{(n)}\chi_{(n)}+m_{(n-1,1)}\chi_{(n-1,1)}$
 \item[2)]
$c_n(A)\ge 2^{n-2}.$
\end{itemize}
   \end{remark}

   \begin{proof}
Let $\lambda=(\lambda_1, \lambda_2, \ldots)\vdash n$ be a partition
of $n$  such that $n-\lambda_1\ge 2.$ This says that either the
first column of $\lambda$ has at least three boxes or the first two
columns of $\lambda$ have at least two boxes each. Hence, if
$f_\lambda$ is an highest weight vector associated to $\lambda$,
either $f$ is alternating on three variables or $f$ is alternating
on two distinct pairs of variables. In both cases every monomial of
$f_\lambda$ evaluated in $A$ contains at least two subwords equal
to $a^2$. Hence $f_\lambda\in Id(A)$ and this implies that
$\chi_\lambda$ appears with zero multiplicity in the decomposition
of $\chi_n(A)$. It follows tha
$$
\chi_n(A)= m_{(n)}\chi_{(n)}+m_{(n-1,1)}\chi_{(n-1,1)}
$$
is the decomposition of $\chi_n(A)$ into irreducibles.

In order to prove 2) we compute the multiplicity $m_{(n)}$ in $\chi_n(A)$.

 Let
$w(L_a,R_a)\in End(A)$ be a word in $L_a$ and $R_a$ of length
$n-2$. Clearly $a^2v(L_a,R_a)=v(L_a,R_a)(a^2)$ is the evaluation of
an highest weight vector associated to the partition $(n)$ which is
not an identity of $A$. Since there are $2^{n-2}$ distinct such
words, we get $2^{n-2}$ highest weight vectors which are linearly
independent mod $Id(A)$. Thus since $\deg \chi_{(n)}=1$, from $\chi_n(A)=
m_{(n)}\chi_{(n)}+m_{(n-1,1)}\chi_{(n-1,1)},$ we have that
$c_n(A)\ge 2^{n-2}.$
 \end{proof}

Next we shall compute the decomposition of the cocharacter
$\chi_n^T(A)$ for a fixed arrangement $T$  of the parentheses  of
$P_n$.

We have the following

\begin{pro} \label{alpha}  For any arrangement $T$ of the parentheses in $P_n$ we have

\begin{equation} \label{eq1}
\chi_n(A)^T= \chi_{(n)}+2\chi_{(n-1,1)}.
\end{equation}

\end{pro}

    \begin{proof}  If $P_n^T(A) \ne 0$ then any
monomial of $P_n^T$ is of the form
$$
x_{\sigma(1)}x_{\sigma(2)}T_{1,x_{\sigma(3)}}\dots T_{n-2,
x_{\sigma(n)}} \quad (\mbox{mod}\ Id(A)),
$$
where $T_{j,x_i}=L_{x_i}$ or $T_{j,x_i}=R_{x_i}$, for any $i,j$.

It follows that, mod $Id(A)$, the highest weight vectors
corresponding to standard tableaux of shape $(n-1,1)$ are
$$
g_0(x_1, x_2)=(\overline x_1\overline x_2)T_{x_1}\dots T_{x_1}
$$
and
$$
g_i(x_1, x_2)=(\overline x_1 x_1)T_{1,x_1}\dots T_{i-1,x_1}\overline
T_{i,x_2}T_{i+1,x_1} \dots T_{n-2,x_1}, \quad  1\le i \le n-2.
$$

Recall that the symbol \ $\bar{}$\  over two or more variables of a polynomial
means that the polynomial is alternating on these variables.

We claim that for any $1 \le i,j \le n-2$ the elements $g_i(x_1,
x_2)$ and $g_j(x_1, x_2)$ are linearly dependent mod $Id(A)$. In
fact, since any word containing two subwords equal to $a^2$ is zero
in $A$, in a non-zero evaluation $\varphi$ we must set
$\varphi(x_1)=a$ and $\varphi(x_2)= a^2v(L_a,R_a)$, for some
monomial $v(L_a,R_a)\in End(A)$.

We get

$$
\varphi(g_i(x_1,x_2))=\varphi(g_j(x_1,x_2))=
- a^2v(L_a,R_a)R_aT_{1,a}\dots T_{n-2,a},
$$
and the claim is established.

Next our aim is to prove that the polynomials $g_0(x_1, x_2)$ and
$g_1(x_1, x_2)$ are linearly independent mod $Id(A)$.
In fact suppose that
$\alpha g_0(x_1, x_2)+\beta g_1(x_1, x_2)$ is an identity of $A$,
for some $\alpha, \beta \in F$.
If  we consider the evaluation
$\varphi(x_1)=a$ and $\varphi(x_2)=a^2$, we get
$$
\alpha g_0(a, a^2)+\beta g_1(a, a^2)=\alpha a^2L_aT_{1,a}\dots T_{n-2,a} -
(\alpha +\beta) a^2R_aT_{1,a}\dots T_{n-2,a},
$$
and the right hand side is zero only if $\alpha=\beta=0.$

We have proved that $\chi_{(n-1,1)}$ appears with multiplicity $2$
in the decomposition of $\chi_n(A)^T$. Since $m^T_{(n)}=1$ we get
that $ \chi_n(A)^T= \chi_{(n)}+2\chi_{(n-1,1)} $ and the
proposition is proved.
\end{proof}
 \smallskip

Next for every real number  between $0$ and $1$ we shall construct a
quotient algebras of $A$. To this end we keep in mind the
terminology of the previous section.

We are going to associate to every finite word in the alphabet $\{0,
1\}$ a monomial in $End(A)$ in left and right multiplications: if
$u(0,1)$ is such a word we associate to $u$ the monomial
$u(L_a,R_a)$ obtained by substituting $0$ with $L_a$ and $1$ with
$R_a$.

Let $\alpha$ be a real number, $0 < \alpha < 1,$ and let $w_\alpha$ be a Sturmian or periodic infinite
word in the alphabet $\{0,1\}$ whose slope is
$\pi(w_\alpha)=\alpha$.

Let $I_\alpha$ be the ideal of the algebra $A$ generated by
the elements $a^2u(L_a,R_a)$ where $u(0,1)$ is not a subword of the
word $w_\alpha$.

Let $A_\alpha=A/I_\alpha$ denote the corresponding quotient
algebra and let $\mathcal{V_\alpha}$ be the variety generated by the
algebra $A_\alpha$.

  We have

\begin{lemma}
For any real number $\alpha$, $0 < \alpha < 1$, the variety $\mathcal{V}_\alpha$
has linear or quadratic growth according as $w_\alpha$ is a periodic or a Sturmian word.
\end{lemma}

\begin{proof}
We are going to find an upper and a lower bound of the codimensions
of the algebra $A_\alpha$. To this end we start from the
decomposition of the cocharacter of $A$ given in (\ref{eq1}).

Let $n\ge 3$ be any integer and let $u(0,1)$ be a subword of the word $w_\alpha$ of length $n-1$.
We may clearly assume that $0$ is the leftmost symbol of such word and,
so, we write $u(0,1)=0v(0,1)$ for some subword $v(0,1)$ of $w_\alpha$.

Since $u$ and $v$ are subwords of $w_\alpha$, $a^2v(L_a,R_a),a^2L_av(L_a,R_a)\not\in I_\alpha$.
This implies that the polynomial
$$
g_0(x_1, x_2)=(\overline x_1 \overline x_2)v(L_{x_1},R_{x_1})
$$
is not an identity of the algebra $A_\alpha$.
In fact, recall that the evaluation $\varphi(x_1)=a, \varphi(x_2)=a^2$ gives
$\varphi(g_0(x_1, x_2))=a^2L_av(L_a,R_a)-a^2R_av(L_a,R_a)\not\in I_\alpha.$

Since the word $u(0,1)$ was an arbitrary subword of the word $w_\alpha$ of length $n-1$,
this says that, for any corresponding arrangement $T$ of the parentheses in
\begin{equation} \label{eq2}
\chi_n(A_\alpha)^T=
\chi_n + m^T_{(n-1,1)}\chi_{n-1,1}
\end{equation}
we must have $m^T_{(n-1,1)}>0.$
Moreover compare the last equality with (\ref{eq1}) and recall that, since $A_\alpha$ is
a quotient algebra of $A,$ the multiplicities in
$\chi_n^T(A_\alpha)$ are bounded by the multiplicities in
$\chi_n^T(A)$. It follows that $0< m^T_{(n-1,1)} \le 2.$

Now, the different arrangements of the parentheses in nonzero words
of length $n$ in $A_\alpha$ correspond to the subwords of $w_\alpha$
of length $n$. Recalling that
$\mbox{Comp}_{w_\alpha}(n)=\mbox{constant}$ \ or
$\mbox{Comp}_{w_\alpha}(n)=n+1$ according as $w_\alpha$ is periodic
or Sturmian, respectively it follows that their number is bounded by
a constant in case $\alpha$ is rational (i.e., $w_\alpha$ is
periodic) and by a linear function of $n$ in case $\alpha$ is
irrational (i.e., $w_\alpha$ is Sturmian).

Since $\deg \chi_{(n)}=1$ and $\deg \chi_{(n-1,1)}= n-1$, from
(\ref{eq2}) and the above discussion we can find constants $C_1,C_2$
such that
for any $n$ we have
$$
C_1n \le c_n(A_\alpha)\le C_2n,
$$
if $\alpha$ is rational,
and
$$
C_1n^2 \le c_n(A_\alpha)\le C_2n^2
$$
if  $\alpha$ is irrational.

Recalling that the growth of $\mathcal{V}_\alpha$ is the growth of
the sequence $c_n(A_\alpha)$ the proof of the lemma is complete.
\end{proof}

\begin{pro} \label{alpha} For $0 < \alpha < \beta < 1,$ the variety
$\mathcal{V_\alpha}\cap \mathcal{V_\beta}$ is nilpotent.
\end{pro}

\begin{proof}
Let $K_n(w_\gamma)$ denote the set of different subwords of length $n$ of a word $w_\gamma.$
now, the slope of the words $w_\alpha$ and  $w_\beta$ is equal to $\alpha$
and $ \beta$, respectively. Since $\alpha \ne \beta$, by Theorem
\ref{lothaire} there exist $m$ such that for any $n\ge m$ the
intersection $K_n(w_\alpha)\cap K_n(w_\beta)$ is the empty set.
In
particular there exist $ m$ such that any word $u(0, 1)$ of
length $m$ is not a subword either of the word  $w_\alpha$ or of the
word $ w_\beta.$

Let for instance $u(0, 1)$  be a word of length $m$ which is not a subword of the
word $w_\alpha,$  and consider the monomial $ y_1y_2u(L_x,R_x).$
Construct the multilinear element $y_1y_2\overline{u}$ on $y_1, y_2, x_1,
\ldots, x_m$ where $\overline{ u }$ is obtained by substituting
$x_1, \ldots, x_m $ instead of $x $ inside $u(L_x, R_x).$ Hence $
y_1y_2\overline{u} \equiv 0$ is an identity of the variety
$V_\alpha.$ It follows that $y_1y_2\overline{u}\equiv 0$ is also an
identity of $V_\alpha \cap V_\beta,$ and so $c_{m+2}(V_\alpha \cap
V_\beta) = 0.$
  From this it follows that
$P^T_{m+2}(\mathcal{V_\alpha}\cap \mathcal{V_\beta})=0$ for any
arrangement of the parentheses $T$ and the variety
$\mathcal{V_\alpha}\cap \mathcal{V_\beta}$ is nilpotent.
\end{proof}

We can now prove the main result of this note.

\begin{theorem} \label{mainth}
Over a field of characteristic zero there are countable many almost nilpotent metabelian varieties
of at most linear growth and uncountable many almost nilpotent metabelian varieties of at most quadratic growth.
\end{theorem}

\begin{proof}
Recall that by  \cite[Theorem 1]{mish-val-exp2} every non-nilpotent variety has an almost nilpotent subvariety.
Hence for any real number $\alpha$, $0<\alpha <1$, the variety $\mathcal{V_\alpha}$ contains
an almost nilpotent subvariety. Let  $\mathcal{U_\alpha}$ be such subvariety.
Since $c_n(\mathcal{U_\alpha})\le c_n(\mathcal{V_\alpha})$, then $c_n(\mathcal{U_\alpha})\le Cn$ or
$c_n(\mathcal{U_\alpha})\le Cn^2$, according as $\alpha$ is rational or irrational, respectively.
Hence $\mathcal{U}_\alpha$ has at most quadratic growth.

Now, by Proposition \ref{alpha} for any $0<\alpha<\beta<1$
$\mathcal{U_\alpha}\ne \mathcal{U_\beta}$, and this says that there
are countable many almost nilpotent metabelian varieties of at most
linear growth  and uncountable many almost nilpotent metabelian
varieties of at most quadratic growth.
\end{proof}

\end{document}